\documentclass[11pt]{article}
\usepackage{amsfonts,amssymb,amsmath,latexsym,xcolor,epsfig}
\usepackage[mathscr]{euscript}

\usepackage{hyperref}
\usepackage{amsthm}

\newcommand{\inj}{\operatorname{inj}}
\newcommand{\ind}{\operatorname{ind}}
\newcommand{\tree}{\operatorname{tree}}

\newtheorem{theorem}{Theorem}
\newtheorem{lemma}[theorem]{Lemma}
\newtheorem{proposition}{Proposition}

\newtheorem{corollary}[theorem]{Corollary}

\newtheorem{definition}{Definition}

\def\qed{\ifvmode\mbox{ }\else\unskip\fi\hskip 1em plus 10fill$\Box$}
\input{epsf}

\makeatletter
\def\Ddots{\mathinner{\mkern1mu\raise\p@
\vbox{\kern7\p@\hbox{.}}\mkern2mu
\raise4\p@\hbox{.}\mkern2mu\raise7\p@\hbox{.}\mkern1mu}}
\makeatother

\title{A short proof of the equivalence of left and right convergence for sparse graphs}
\author{L\'aszl\'o Mikl\'os Lov\'asz\thanks{Department of Mathematics,
    Massachusetts Institute of Technology,
    Cambridge, MA 02139-4307.
    Email: {\tt lmlovasz@math.mit.edu}. This work was partially done during the author's visit to the group of Dan Kr\'al' at the University of Warwick; the visit was partially supported by the European Research Council under the European Union's Seventh Framework Programme (FP7/2007-2013)/ERC grant agreement no.~259385. Another part of this
work was done with the support of Jacob Fox's NSF CAREER award DMS
1352121}
}
\date{\today}
\begin{document}
\maketitle

\begin{abstract}
There are several notions of convergence for sequences of bounded degree graphs. One such notion is left convergence, which is based on counting neighborhood distributions. Another notion is right convergence, based on counting homomorphisms to a target (weighted) graph. Borgs, Chayes, Kahn and Lov\'asz showed that a sequence of bounded degree graphs is left convergent if and only if it is right convergent for certain target graphs $H$ with all weights (including loops) close to $1$. We give a short alternative proof of this statement. In particular, for each bounded degree graph $G$ we associate functions $f_{G,k}$ for every positive integer $k$, and we show that left convergence of a sequence of graphs is equivalent to the convergence of the partial derivatives of each of these functions at the origin, while right convergence is equivalent to pointwise convergence. Using the bound on the maximum degree of the graphs, we can uniformly bound the partial derivatives at the origin, and show that the Taylor series converges uniformly on a domain independent of the graph, which implies the equivalence.
\end{abstract}

\section{Introduction}

The theory of graph limits has been extensively developing in recent years, primarily in the dense case \cite{BCLSV1,BCLSV2}. For sequences of sparse graphs, in particular, sequences of graphs with bounded degree, much less is known. A notion of convergence, which we will refer to as left convergence, was first defined by Benjamini and Schramm \cite{BeSch}. One of the major open problems on bounded degree graph sequences, the conjecture of Aldous and Lyons \cite{AL} on left convergent sequences of graphs, is very closely related to Gromov's question about whether all countable discrete groups are sofic \cite{GROM}. In a certain sense, left convergence is too rough to ``distinguish'' graphs which should be different. Because of this, other, finer notions of convergence for sequences of bounded degree graphs \cite{BolRio,BCG,HaLoSze} have been proposed. The full picture of their mutual relations is not completely understood. Borgs, Chayes, Kahn and Lov\'asz \cite{BCLK} prove that a sequence of bounded degree graphs is left convergent if and only if it is ``right convergent'' for a certain set of target graphs. We provide an alternative, shorter proof of this statement, which we believe also gives some new and useful insights.

Here and throughout this paper we use the following setting: $G_n=(V(G_n),E(G_n))$, is a sequence of graphs with degrees bounded uniformly by a constant $D$, and $v(G_n)=|V(G_n)| \rightarrow \infty$. Given two simple graphs $G$ and $H$, we let $\hom(G,H)$ denote the number of maps $V(G) \rightarrow V(H)$ such that any edge of $G$ is mapped to an edge of $H$. We define $\inj(G,H)$ as the number of maps that are injective on the set of vertices, and $\ind(G,H)$ as the number of injective maps that are isomorphisms between $G$ and the subgraph of $H$ induced by the images of the vertices in $G$.

Suppose now that $H$ is a weighted graph, with vertex weights $w_i$ and edge weights $w_{ij}$. We think of each pair of vertices as having a weight, perhaps equal to $0$. We define $\hom(G,H)$ as the sum over all maps $f: V(G) \rightarrow V(H)$ of the product 
\[\prod_{v \in G} w_{f(v)} \prod_{(uv) \in E(G)} w_{f(u)f(v)}. \]

Note that if all vertices in $H$ have weight $1$ and all edges have weight $0$ or $1$, then $\hom(G,H)$ is equal to $\hom(G,H')$ where $H'$ is the unweighted graph on the vertex set of $H$ formed by the edges with weight $1$.
We also define $t(G,H)$ as the \emph{average} of this product over all maps, this simply divides it by a factor of $v(H)^{v(G)}$.

One important notion of convergence is the following, defined by Benjamini and Schramm \cite{BeSch}:

\begin{definition}[Left convergence]
Given a sequence of graphs $(G_n)$ with all degrees bounded by a positive integer $D$, we say that the sequence is \textit{left convergent} if for any connected graph $F$, the limit 
\[\lim_{n \rightarrow \infty} \frac{\ind(F,G_n)}{v(G_n)}\]
exists. It is well known that we obtain an equivalent definition if we replace $\ind$ with $\inj$ or $\hom$.
\end{definition}

Note that the original definition by Benjamini and Schramm involved looking at the distribution of the $r$-neighborhoods of vertices for any $r$, but it is easy to see that the two notions are equivalent. 

We also examine the notion of right convergence:

\begin{definition}[Right convergence]
Given a sequence of graphs $(G_n)$ with bounded degrees, we say that the sequence is \textit{right convergent with soft-core constraints}, or simply soft-core right convergent, if for any weighted target graph $H$ that is complete and has positive weights (including on loops), the limit 
\[\lim_{n \rightarrow \infty} \frac{\log\hom(G_n,H)}{v(G_n)}\] 
exists. If the limit also exists for all graphs $H$ with nonnegative edge weights, we will refer to it as right convergence with hard-core constraints, or hard-core right convergence.

\end{definition}

If we consider $t(G_n,H)$ instead of $\hom(G_n,H)$, we obtain an equivalent definition, this simply decreases each term in the sequence by the same constant, $\log v(H)$. Given a sequence $G_n$ of bounded degree graphs, if for each $n$ we add or delete $o(v(G_n))$ edges in $G_n$, neither soft-core right convergence, nor left convergence is affected. However, hard-core right convergence can change. Borgs, Chayes, and Gamarnik \cite{BCG} show that if a sequence is soft-core right convergent, then one can delete $o(v(G_n))$ edges from $G_n$ to make it hard-core right convergent. In particular, this implies that if every hard-core right convergent sequence is left convergent, then every soft-core right convergent sequence is left convergent.

Borgs, Chayes, Kahn and Lov\'asz \cite{BCLK} proved that hard-core right convergence implies left convergence, and left convergence implies right convergence for a subset of target graphs $H$, where the set depends on $D$, the bound on the maximum degree.
Their proof considers target graphs $H$ with zero weights. The proof in this note only considers target graphs with positive weights, and thus gives a more direct proof that (soft-core) right convergence implies left convergence. 
Specifically, we provide a new short proof of the following theorem:

\begin{theorem} \label{main}
For any sequence $(G_n)$ of graphs with all degrees bounded by $D$, the following are equivalent:
\begin{enumerate}
\item $(G_n)$ is left convergent.
\item For every target graph $H$ with each vertex and edge weight between $e^{-(4eD)^{-1}}$ and $e^{(4eD)^{-1}}$, the sequence $\frac{1}{v(G_n)}\log t(G,H)$ is convergent.
\item For each $k$, there exists an $\epsilon_k$ such that for any target graph $H$ on $k$ vertices with each weight between $e^{-\epsilon_k}$ and $e^{\epsilon_k}$, the sequence $\frac{1}{v(G_n)	}\log t(G,H)$ is convergent.
\end{enumerate}
\end{theorem}

We will now introduce an alternative way of looking at right convergence that provides a more natural setting for the proof.

\begin{definition}
Given a graph $G$, and a positive integer $k$, we define $X(G,k) \in \mathbb{R}^{k+k^2}$ as follows. Take a random coloring of $V(G)$ with $k$ colors. For a fixed coloring $C$, we define $X(G,C)_i$ as the proportion of vertices with color $i$. That is, we take the number of vertices with color $i$ and divide it by $v(G)$. We define $X(G,C)_{i,j}$ as the number of edges between color class $i$ and $j$, divided by $v(G)$. 
Note that $X(G,C)_{i,j}=X(G,C)_{j,i}$. We now define the random variables $X_i=X(G,k)_i$ and $X_{i,j}=X(G,k)_{i,j}$ 
as the values of $X(G,C)$ where the coloring is uniformly random among all 
$k^v(G)$-colorings of the vertices. 
Note that $X \in [0,1]^k \times [0,D]^{k^2}$, where $D$ is the bound on the degree.
\end{definition}

We will examine the cumulant generating function of this variable.

\begin{definition}\label{fdefi}
Given a positive integer $k$ and $\lambda \in \mathbb{R}^{k+k^2}$, we define $f_{G,k}(\lambda)$ as the (normalized) cumulant generating function 
\[f_{G,k}(\lambda)=\frac{1}{v(G)} \log \mathbb{E}(e^{\langle \lambda,v(G) X(G,k) \rangle}).\]
\end{definition}

For $\lambda \in \mathbb{R}^{k+k^2}$, define the weighted graph $H_\lambda$ on the vertex set $[k]$ with vertex $i$ having weight $e^{\lambda_i}$ and edge $ij$ having weight $e^{\frac{1}{2}(\lambda_{ij}+\lambda_{ji})}$. Then it is not difficult to see that
\[f_{G,k}(\lambda)=\frac{1}{v(G)}\log t(G,H_\lambda).\]
Thus, given a graph sequence $G_n$, the pointwise convergence of $f_{G_n,k}(\lambda)$ for $\lambda \in S$ for some $S \subset \mathbb{R}^{k+k^2}$ is equivalent to right convergence for all target graphs $H_\lambda$ for $\lambda \in S$.

With this notation, Theorem \ref{main} is equivalent to the following theorem.

\begin{theorem} \label{main2}
For any sequence $(G_n)$ of graphs with all degrees bounded by $D$, the following are equivalent:
\begin{enumerate}
\item $(G_n)$ is left convergent.
\item For every $\| \lambda \|_\infty <\frac{1}{4eD}$, the sequence $f_{G_n,k}(\lambda)$ is convergent.
\item For every $k$, there exists an open neighborhood $S_k$ of the origin in $\mathbb{R}^{k+k^2}$ such that for any $\lambda \in S_k$, the sequence $f_{G_n,k}(\lambda)$ is convergent.
\end{enumerate}
\end{theorem}

Note that trivially $(2) \Rightarrow (3)$. Our proof strategy is the following. We show that for a fixed graph $G$ the functions $f_{G,k}$ are analytic around zero, and for a fixed $k$ the Taylor series around the origin converges uniformly on the region $\|\lambda\|_\infty < (4eD)^{-1}$. Then we show that left convergence is equivalent to the convergence of the partial derivatives of $f_{G_n,k}$ at the origin for all $k$ and all partial derivatives. The equivalence then follows from standard analysis arguments.

In order to analyze the partial derivatives, we use the notion of the joint cumulant of random variables. If we have $l$ random variables $Z_1,Z_2,...,Z_l$, their joint cumulant $\kappa(Z_1,...,Z_l)$ is defined as follows. Let $g(\lambda_1,...,\lambda_l)=\log \mathbb{E}(\exp(\sum_{i=1}^l \lambda_i Z_i))$ and take the partial derivative of $g$ once according to each variable, at the origin. In particular, $\kappa(X)$ is the expected value $\mathbb{E}(X)$ and $\kappa(X,Y)$ is the covariance $\sigma(X,Y)=\mathbb{E}(XY)-\mathbb{E}(X)\mathbb{E}(Y)$. The mapping $\kappa$ is a multilinear function of the variables, and if we let $\pi$ run over the partitions of the set $[l]$, we have the formula 
\begin{equation} \label{cumulexpr}
\kappa(Z_1,Z_2,...,Z_l)=\sum_\pi(|\pi|-1)!(-1)^{|\pi|-1}\prod_{B \in \pi}\mathbb{E}\left(\prod_{i \in B}Z_i\right)
\end{equation}

We also define the $r$-th cumulant $\kappa_r(Z)=\kappa(Z,Z,...,Z)$ where $Z$ appears $r$ times.

\section{Analyticity and the domain of convergence}

In this section, we will show that for each $k$, and for each graph $G$ with maximum degree at most $D$, the Taylor series of $f_{G,k}$ converges on an open neighborhood of the origin, uniformly for fixed $D$.

\begin{proposition} \label{convdom}
Suppose $G$ is a graph with maximum degree at most $D$. Then $f_{G,k}$ is analytic around $0$, and at any point in the domain $\| \lambda \|_\infty < (4eD)^{-1}$, or even on the set $\|\lambda\|_\infty \le c$ for any $c<(4eD)^{-1}$, the Taylor series centered at the origin converges uniformly across all graphs with maximum degree $D$.
\end{proposition}

In order to prove this proposition, we first need to state an auxiliary result. Specifically, we will use Theorem $9.3$ in \cite{FMN}, which states the following: 
\begin{theorem} \label{cumulbound} Let $\{Y_\alpha \}_{\alpha \in W}$ be a family of random variables with finite moments, and assume there is a graph $L$ on $W$ with maximum degree $\Delta$ that has the following property: If $W_1$ and $W_2$ are disjoint subsets of $W$ with no edges between them, then $\{Y_\alpha \}_{\alpha \in W_1}$ and $\{Y_\alpha \}_{\alpha \in W_2}$ are independent. Assume that each $\|Y_\alpha\|_r \le A$, where $\|X\|_r=(\mathbb{E}|X|^r)^{1/r}$. Let $Y=\sum_\alpha Y_\alpha$. Then we have the following bound on the cumulant: \[|\kappa_r(Y)| \le 2^{r-1}r^{r-2}|W|(\Delta+1)^{r-1}A^r.\]

\end{theorem}

The main tool in their proof is the following lemma, implicit in their proof:
\begin{lemma}
Given $r$ random variables $Y_{\alpha_1},Y_{\alpha_2},...,Y_{\alpha_r}$, if we let $H$ be a graph on $[r]$ with the same property as above, then 
\[ | \kappa(Y_{\alpha_1},Y_{\alpha_2},...,Y_{\alpha_r})| \le A^r 2^{r-1} \tree(H),\]
where $\tree(H)$ is the number of spanning trees of $H$. 
\end{lemma}

A rough outline of the proof of Theorem \ref{cumulbound} in \cite{FMN} is as follows. The joint cumulant is a multilinear function, so it is just the sum of $\kappa(Y_{\alpha_1},Y_{\alpha_2},...,Y_{\alpha_r})$ where each $\alpha_i \in W$. Since we can think of $H$ as always having $[r]$ as its vertex set, we can take each possible tree on $[r]$ (we have $r^{r-2}$), and using the bound $\Delta$ on the maximum degree of $L$, we can bound the number of times each possible tree will be a subtree of $H$.

\begin{proof}[Proof of Proposition \ref{convdom}]
Clearly if we fix $G$, $f_{G,k}$ is analytic on some domain around zero. To analyze the domain, we analyze it in a fixed direction. Take any $\lambda_0 \in \mathbb{R}^{k+k^2}$ such that $\|\lambda_0 \|_\infty=1$, and consider the function $g:\mathbb{R} \rightarrow \mathbb{R}, g(z)=f_{G,k}(z \lambda_0)$. Let $Y=\langle \lambda_0,v(G) X \rangle$. We can write $Y=\sum_\alpha Y_\alpha$, where $\alpha$ is either a vertex or an edge of $G$, and gives the contribution of that vertex or edge to $\langle \lambda_0,v(G)X \rangle$. We define the graph $L$ with vertex set $W=V(G) \cup E(G)$, and we connect each $v \in V(G)$ with its incident edges, and each $e \in E(G)$ with its two incident vertices and its incident edges. The maximum degree of $L$ will be at most $\Delta=2D$. We also have that for any $l \ge 1$, $\| Y_\alpha \|_l \le \| Y_\alpha \|_\infty \le 1$. Using Theorem \ref{cumulbound}, 
\[|g^{(l)}(0)| = |\kappa_r(Y_0)| \le 2^{l-1}l^{l-2}(v(G)+|E(G)|)|(2D+1)^{l-1}.\]
Since $|E(G)| \le v(G)D/2$, \[\limsup_{l \rightarrow \infty}\Bigg( \frac{|g^{(l)}(0)|}{l!}\Bigg)^{1/l} \le \limsup_{l \rightarrow \infty} \Bigg(\frac{2^{l-2}l^{l-2}v(G)(D+2)(2D+1)^{l-1}}{l!}\Bigg)^{1/l}=4eD ,\] so the Taylor series converges for $z \in \mathbb{R}$, $|z| < 1/(4eD)$. Because of the uniform bound on the partial derivatives, the convergence is also uniform across all graphs (of bounded degree $D$) for $|z| \le c$.

\end{proof}

\section{Left convergence and the partial derivatives}

In this section, we will prove the equivalence of left convergence with the convergence of the partial derivatives of $f_{G,k}$ at the origin. We will use the following notion, closely related to $\inj$. Let $F$ be a multigraph with $l$ labeled edges $e_1,e_2,...,e_l$ (for some positive integer $l$), and no isolated vertices. For a graph $G$, we define $i(F,G)$ as the number of ways of choosing a sequence of $l$ edges of $G$ (that is, a multiset of $l$ edges of $G$ labeled from $1$ to $l$), such that the mulitgraph induced by these $l$ edges is isomorphic to $F$. Let $\mathscr{F}^*_l$ be the collection of multigraphs with $l$ edges labeled from $1$ to $l$, but vertices unlabeled, with no isolated vertices. Let $\mathscr{F}_l$ be the subset of $\mathscr{F}^*$ consisting of just the connected multigraphs. 

We observe that left convergence is equivalent to saying that the limit
\[\lim_{n \rightarrow \infty} \frac{i(F,G_n)}{v(G_n)}\]
exists for any $F \in \mathscr{F}_l$ for all $l>0$. Indeed, for any $F \in \mathscr{F}_l$, $i(F,G_n)=i(F',G_n)$, where $F'$ is obtained from $F$ by removing parallel edges. Then $i(F',G)=\inj(F',G)$, unless $F'$ is a single edge, in which case $i(F',G)=\inj(F',G)/2$. Thus, convergence of $i(F,G_n)/v(G_n)$ for all $F \in \mathscr{F}_l$ is equivalent to convergence of $\inj(F,G_n)/v(G_n)$ for all connected graphs $F$ with at most $l$ (unlabeled) edges (note $\inj(F,G)=0$ for $F$ with parallel edges).

\begin{proposition} \label{indep}
Suppose we have a graph $G$ and a positive integer $k$. Consider the function $f_{G,k}:\mathbb{R}^{k+k^2} \rightarrow \mathbb{R}$ defined in \ref{fdefi}.
For any nonnegative integer $l$, the values of the $l$-th partial derivatives at the origin are linear combinations of the values of $i(F,G)/v(G)$ over various choices of $F \in \cup_{l' \le l}\mathscr{F}_{l'}$ (recall that in the definition of $f_{G,k}$ we also divided by $v(G)$), and the coefficients do not depend on $G$. Moreover, if $k \ge 2l$, then this is a bijection, that is, if two graphs $G_1$ and $G_2$ have different values of $i(F,G_1)/v(G_1)$ and $i(F,G_2)/v(G_2)$ for some connected graph $F$ on at most $l$ edges, then their $l$-th partial derivatives will be different.

\end{proposition}


\begin{proof}[Proof of Proposition \ref{indep}]
For ease of presentation of the proof, we will assume that each variable in the partial derivative corresponds to an edge, it will be clear that the proof also works if some of the variables correspond to vertices. In this case, we will only need $F \in \mathscr{F}_l$. We know that the partial derivative according to the variables $\lambda_{i_1,j_1},\lambda_{i_2,j_2},..., \lambda_{i_l,j_l}$ is equal to $\frac{1}{v(G)}\kappa(v(G)X_{i_1,j_1},v(G)X_{i_2,j_2},...,v(G)X_{i_l,j_l})$. Suppose the pairs $(i_1,j_1),(i_2,j_2),...,(i_l,j_l) \in [k]^2$, form the (multi)graph $\widetilde{J}$. We think of $\widetilde{J}$ as a multigraph on $[k]$ with $l$ edges whose edges are labeled $1$ to $l$. Define $\kappa_G(\widetilde{J})$ as the joint cumulant $\kappa(v(G)X_{i_1,j_1},v(G)X_{i_2,j_2},...,v(G)X_{i_l,j_l})$ . Note that if we permute the vertices of $\widetilde{J}$, $\kappa_G(\widetilde{J})$ does not change (this is the same as permuting the colors).

Each $\widetilde{J}$ has a representative $J$ in $\mathscr{F}_l^*$ (after removing isolated vertices and unlabeling the remaining vertices), and because permuting the vertices does not change the expression, $\kappa_G(\widetilde{J})$ depends only on the representative $J$ (and on $k$, which is fixed), so with a slight abuse of notation we will write $\kappa_G(J)$ for $J \in \mathscr{F}^*_l$. If $k \ge 2l$, then each graph in $\mathscr{F}^*_l$ will have a corresponding graph on the vertex set $[k]$ (after adding isolated vertices).

We have 
\[\kappa(v(G)X_{i_1,j_1},v(G)X_{i_2,j_2},...,v(G)X_{i_l,j_l})=\sum_{e_1,e_2,...,e_l \in E(G)}\kappa(Y_{i_1,j_1,e_1},Y_{i_2,j_2,e_2},...,Y_{i_l,j_l,e_l}).\]

Here $Y_{i_p,j_p,e_p}$ is the indicator random variable of whether the two endpoints of $e_p$ are colored with colors $i_p$ and $j_p$. Let us look at a term $\kappa(Y_{i_1,j_1,e_1},Y_{i_2,j_2,e_2},...,Y_{i_l,j_l,e_l})$. For simplicity of notation, let $Z_p=Y_{i_p,j_p,e_p}$. It is not difficult to see from the definition of the joint cumulant that if $Z_1,Z_2,...,Z_l$ can be partitioned into two groups that are independent of each other, then $\kappa(Z_1,Z_2,...,Z_l)$ is zero. This happens if $e_1,e_2,...,e_l$ do not form a connected subgraph. Otherwise, we can write \eqref{cumulexpr}
\[\kappa(Z_1,Z_2,...,Z_l)= \sum_\pi (|\pi|-1)!(-1)^{|\pi|-1}\prod_{B \in \pi}\mathbb{E}(\prod_{p \in B}Z_p),\]
 


We now introduce additional notation to facilitate the analysis of this expression. Let $E$ be any multigraph with $l$ labeled edges $f_1,f_2,...,f_l$. Denote $\mathbb{E}(\prod_p Y_{i_p,j_p,f_p} )$ as $x(E,J)$. Note that if $E$ is disconnected, the disjoint union of $E_1$ and $E_2$, then $x(E,J)=x(E_1,J)x(E_2,J)$. We can define this value for any $J$ that comes from a graph on the vertex set $[k]$. Now, given a partition $\pi$ of $[l]$, we define $E_\pi$ as follows. For each $B \in \pi$, take the subgraph of $E$ formed by the edges with labels in $B$, and the vertices that are endpoints of at least one of these edges. Then take $E_\pi$ as the disjoint union of these graphs, where each edge keeps its original label from $E$. 

Let $F$ denote the subgraph formed by the edges $e_1,e_2,...,e_l$ (with edges labeled). Using the above notation, we have that \[\kappa(Z_1,Z_2,...,Z_l)=\sum_\pi(|\pi|-1)!(-1)^{|\pi|-1}x(F_\pi,J).\]
We can see that the previous expression depends only on $F$, so define for any graph $F \in \mathscr{F}_l^*$
\[\kappa(F,J):=\sum_\pi(|\pi|-1)!(-1)^{|\pi|-1}x(F_\pi,J).\]

Then (recalling that $J$ is the graph formed by the pairs $(i_p,j_p)$) we have
\[\kappa(v(G)X_{i_1,j_1},...,v(G)X_{i_l,j_l})= \sum_{F \in \mathscr{F}_l} i(F,G) \sum_{\pi } (|\pi|-1)!(-1)^{|\pi|-1} x(F_\pi,J)=\sum_{F \in \mathscr{F}_l}i(F,G)\kappa(F,J).\]
This proves the first part of the Proposition.

Let us prove the second part. Since $k \ge 2l$, each graph $J \in \mathscr{F}^*_l$ has a corresponding graph on $[k]$, so we just need to show that if we have two graphs $G$ and $G'$ such that $i(F,G)\ne i(F,G')$ for some $F \in \mathscr{F}_l$, then $\kappa_G(J) \ne \kappa_{G'}(J)$ for some $J \in \mathscr{F}^*_l$. Let $K$ be the matrix with columns indexed by $F \in \mathscr{F}_l$, rows indexed by $J \in \mathscr{F}_l^*$, and entries $\kappa(F,J)$. Let $E$ be the matrix with rows and columns indexed by $\mathscr{F}_l^*$, and the entry in column $F$ and row $J$ is $x(F,J)$. Let $P$ be the matrix with columns indexed by $\mathscr{F}_l$, rows indexed by $\mathscr{F}_l^*$, and the entry in column $F$ and row $F'$ is $(-1)^{|\pi|-1}(|\pi|-1)!$ if $F'=F_\pi$ for some $\pi$ partition of $[l]$, and $0$ otherwise. Then $K=EP$. If $u\in \mathscr{R}^{\mathscr{F}_l^*}$ is the vector with entries $\kappa_G(J)$, and $w \in \mathscr{R}^{\mathscr{F}_l}$ is the vector with entries $i(F,G):F \in \mathscr{F}_l$, then $u=Kw$. Thus, we need to show that the columns of $K$ are independent, and we will do this by showing that $E$ is invertible, and that the columns of $P$ are independent.

First, let us look at $E$. Its entries are $x(F,J)$, which is zero if it is not possible for the edges of $F$ to map to $J$. In particular, $x(F,J)$ is zero if $J$ has more non-isolated vertices than $F$, or if the number of non-isolated vertices is equal and the two graphs are not isomorphic (taking the edge labels into consideration). However, it is easy to see that the diagonal entries $x(F,F)$ are nonzero. Thus, if we order $\mathscr{F}^*_l$ such that the number of non-isolated vertices is increasing, then the matrix will have nonzero entries in the diagonal, and zero entries above the diagonal. Thus, $E$ is invertible. To see that $P$ has independent columns, we can just restrict to the rows in $\mathscr{F}_l$, and see that we obtain the identity matrix. So we have shown that $P$ has independent columns, $E$ is invertible, and so $K=EP$ has independent columns. Thus, if two graphs have different values of $i(F,G)$ for some $F \in \mathscr{F}_l$, then at least one of their $l$-th partial derivatives must be different.

\end{proof}

From the proposition, we easily obtain the following corollary:

\begin{corollary} \label{leftpartequiv}
A sequence of graphs $(G_n)$ is left convergent if and only if for any $k$, all partial derivatives of the functions $f_{G_n,k}$ at $0$ converge.
\end{corollary}

We are now ready to prove our main theorem.
\begin{proof}[Proof of Theorem \ref{main}] First, we show that $(1)$ implies $(2)$. By Proposition \ref{convdom}, if we are within the domain  $\| \lambda \|_\infty <(4eD)^{-1}$, then the Taylor series of $f_{G,k}$ converges uniformly across all graphs $G$ with bounded maximal degree $D$. Since the coefficients of the Taylor series are the partial derivatives (multiplied by fixed constants), this immediately implies that if a sequence $(G_n)$ is left convergent, then the values $f_{G,k}(\lambda)=\frac{\log t(G_n,H_\lambda)}{v(G_n)}$ converge if $\|\lambda\|_\infty<(4eD)^{-1}$. Since it is clear that $(2)$ implies $(3)$, it remains to show that $(3)$ implies $(1)$. Since there is a universal bound on the partial derivatives, the Taylor series converges to the function uniformly on an open neighborhood of the origin. Since the terms up to order $k$ approximate the function with error $o(\|\lambda\|_\infty^k)$, if the functions converge on an open neighborhood, the partial derivatives converge. By Corollary \ref{leftpartequiv}, this implies left convergence.
\end{proof}

\section{Concluding Remarks}

We remark that it also makes sense to define left and right convergence for sequences of bounded degree graphs with parallel edges. There are a few different ways to extend the definition of $\hom, \inj$ and $\ind$, however they will produce equivalent notions of left convergence. The notion will also still be equivalent to convergence of $i(F,G_n)/v(G_n)$ for $F \in \mathscr{F}_l$. The definition of $\hom(G,H)$ if $H$ is weighted extends naturally to the case when $G$ has parallel edges, and so does the definition of right convergence, and our proof will still work.

\section*{Acknowledgements}
The author would like to thank Dan Kr\'al' for insightful conversations and for helpful comments, and Jacob Fox for comments on the manuscript.

\bibliography{ref}

\bibliographystyle{amsplain}

\end{document}